\documentclass[12pt]{amsart}
\usepackage{amsmath,amsthm,amssymb}
\usepackage{graphicx}
\usepackage[usenames]{color}
\begin{document}
\newtheorem{thm}{Theorem}
\newtheorem{pro}[thm]{Proposition}
\newtheorem{cor}[thm]{Corollary}
\newtheorem{lem}[thm]{Lemma}
\newtheorem{dfn}[thm]{Definition}
\newtheorem{rem}[thm]{Remark}
\newtheorem{prob}[thm]{Problem}
\newtheorem{exam}[thm]{Example}
\newtheorem{conj}[thm]{Conjecture}
\def\vol{\mathop{\mathrm{Vol}}\nolimits}

\renewcommand{\theequation}{\arabic{section}.\arabic{equation}}
\renewcommand{\labelenumi}{\rm{(\arabic{enumi})}}
\title[are not Hamiltonian volume minimizing]
{Almost all Lagrangian torus orbits in ${\mathbb C}P^n$ are not Hamiltonian volume minimizing}
\author[H. Iriyeh and H. Ono]
{Hiroshi Iriyeh and Hajime Ono}
\date{}
\keywords{Lagrangian torus; toric manifold; Hamiltonian stability; Hamiltonian volume minimizing.\\
\indent
The first author was partly supported by the Grant-in-Aid for Young Scientists (B)
(No.~24740049), JSPS.
The second author was partly supported by the Grant-in-Aid for
Scientific Research (C) (No.~24540098), JSPS}
\subjclass[2010]{Primary 53D12; Secondary 53D10}

\begin{abstract}
All principal orbits of the standard Hamiltonian $T^n$-action on the complex
projective space ${\mathbb C}P^n$ are Lagrangian tori.
In this article, we prove that most of them are not volume minimizing under Hamiltonian
isotopies of ${\mathbb C}P^n$ if the complex dimension $n$ is greater than two,
although they are Hamiltonian minimal and Hamiltonian stable.
%We also discuss the existence of Hamiltonian non-volume minimizing
%Lagrangian torus orbits of compact toric K\"ahler manifolds.
\end{abstract}

\maketitle
%\textcolor{red}{

\section{Introduction}

The classical isoperimetric inequality for a simple closed curve $L$ in $\mathbb R^2$
(resp.\ the unit two-sphere $S^2$) states that
$$
l(L)^2 \geq 4\pi A \ \ \ (\mathrm{resp.}\ l(L)^2 \geq 4\pi A -A^2),
$$
where $l(L)$ is the length of $L$ and $A$ the area of the disc enclosed by $L$.
Moreover, the equality holds if and only if $L$ is a round circle.
In other words, a round circle $L=S^1$ in $\mathbb R^2$ (or $S^2$) has least length
when we deform $L$ in such a way that the enclosed area $A$ is unchanged.
Notice that without this last constraint we can easily reduce the length of $S^1$
by deforming it to the normal direction.

In papers \cite{Oh1990} and \cite{Oh1993}, 
Y.-G. Oh proposed a higher dimensional analogue of such a phenomenon from
the symplectic geometrical viewpoint and introduced several concepts.
Let us review the settings.
Let $(M,\omega,J)$ be a K\"ahler manifold.
A submanifold $L$ of $M$ is said to be {\it Lagrangian} if $\omega|_{TL} \equiv 0$ and
$\dim_{\mathbb R}L=\dim_{\mathbb C}M$.
This condition is equivalent to the existence of an orthogonal decomposition
$$
T_p M= T_p L \oplus J(T_p L)
$$
for any $p \in L$.
Throughout this article all Lagrangian submanifolds are assumed to be connected, embedded,
closed and equipped with the induced Riemannian metric from the ambient manifold $M$.
We denote by ${\rm Vol}(L)$ the volume of $L$ with respect to the metric.

Notice that ${\mathbb R}^2 \cong \mathbb C$ and $S^2$ are one-dimensional K\"ahler manifolds
and a simple closed curve in them is a Lagrangian submanifold.
The constraint that $A$ is constant is generalised to the deformation of a Lagrangian
submanifold $L$ under {\it Hamiltonian isotopies} explained below.
By definition, we have the linear isomorphism defined by
$$
\Gamma(T^\perp L) \ni V \longmapsto \alpha_V := \omega(V,\cdot)|_{TL} \in \Omega^1(L),
$$
where $T^\perp L (\cong J(TL))$ denotes the normal bundle of $L \subset M$ and
$\Omega^1(L)$ the set of all one-forms on $L$.
A variational vector field $V \in \Gamma(T^\perp L)$ of $L$ is called a
{\it Hamiltonian variation} if $\alpha_V$ is exact.
It implies that the infinitesimal deformation of $L$ with the vector field $V$
preserves the Lagrangian constraint.
The following definitions are due to Oh.

\begin{dfn}[\cite{Oh1993}, \cite{Oh1990}] \rm
Let $L$ be a Lagrangian submanifold of a K\"ahler manifold $(M,\omega,J)$.
\begin{enumerate}
\item $L \subset M$ is said to be {\it Hamiltonian minimal} if it satisfies that
\begin{eqnarray*}
\frac{d}{dt} {\rm Vol}(L_t)\Big|_{t=0} = 0
\end{eqnarray*}
for any smooth deformation $\{ L_t \}_{-\epsilon<t<\epsilon}$ of $L=L_0$
with a Hamiltonian variation $V=\frac{dL_t}{dt}|_{t=0}$.
\item Suppose that $L \subset M$ is Hamiltonian minimal.
Then $L$ is said to be {\it Hamiltonian stable} if it satisfies that
\begin{eqnarray*}
\frac{d^2}{dt^2} {\rm Vol}(L_t)\Big|_{t=0} \geq 0
\end{eqnarray*}
for any smooth deformation $\{ L_t \}_{-\epsilon<t<\epsilon}$ of $L=L_0$
with a Hamiltonian variation $V=\frac{dL_t}{dt}|_{t=0}$.
\item $L \subset M$ is said to be {\it Hamiltonian volume minimizing} if
$$
\vol(\phi(L)) \geq \vol(L)
$$
holds for any $\phi \in \mathrm{Ham}_c(M,\omega)$,
which is the set of all compactly supported Hamiltonian diffeomorphisms of $(M,\omega)$.
\end{enumerate}
\end{dfn}

A diffeomorphism $\phi$ of $(M,\omega)$ is called
{\it Hamiltonian} if $\phi$ is the time-one map of the flow
$\{ \phi_H^t \}_{0 \leq t \leq 1}$, $\phi_H^0=id_M$,
of the (time-dependent) Hamiltonian vector field $X_{H_t}$ defined by
a compactly supported Hamiltonian function $H \in C^\infty_c([0,1] \times M)$.
The isotopy $\{ \phi_H^t \}_{0 \leq t \leq 1}$ is called a {\it Hamiltonian isotopy} of $M$.
It is easy to see that $(\phi_H^t)^*\omega=\omega$.
Note that a (time-independent) Hamiltonian vector field on $M$ gives rise to a Hamiltonian variation
of a Lagrangian submanifold $L \subset M$.

At present we know only a few non-trivial examples of Hamiltonian volume minimizing
Lagrangian submanifolds except for special Lagrangian submanifolds;
the real form ${\mathbb R}P^n \subset {\mathbb C}P^n$, \cite{Oh1990},
the product of the great circles in $S^2 \times S^2$, \cite{IOS2003},
and the totally geodesic Lagrangian sphere $S^{2n-1}$ in the complex hyperquadric
$Q_{2n-1}(\mathbb C)$, \cite{IST2013}.

\smallskip

The most fundamental example of symplectic manifolds is
the linear complex space $\mathbb C^n$ equipped with the standard symplectic structure
$\omega_0:=dx_1 \wedge dy_1 +\cdots+ dx_n \wedge dy_n$.
Its standard complex structure and $\omega_0$ define the standard Euclidean metric on
${\mathbb C}^n \cong {\mathbb R}^{2n}$.
We denote by $S^1(b) \subset \mathbb R^2 \cong \mathbb C$
the boundary of a round disc with area $b$ centred at the origin,
i.e., the radius of $S^1(b)$ is $\sqrt{b/\pi}$.
For positive real numbers $b_1,\ldots,b_n>0$ , the {\it product torus}
(or {\it elementary torus}, see \cite{Chekanov1996})
$$
T(\boldsymbol{b})
=T(b_1,\ldots,b_n):=S^1(b_1) \times \cdots \times S^1(b_n) \ \subset \ \mathbb C^n
$$
is a typical example of Lagrangian submanifolds of $\mathbb C^n$.
Here we denote $N(\boldsymbol{b}):=\#\{b_1,\dots,b_n\}$, e.g.,
$N(\boldsymbol{b})=3$ for $\boldsymbol{b}=(1,2,2,4)$.

We can easily check, using the first variation formula (see \cite[p.\ 178]{Oh1993}), that
$L \subset M$ is Hamiltonian minimal if and only if the equation $\delta \alpha_H=0$ holds on $L$,
where $\delta$ and $H$ are the codifferential operator on $L$
and the mean curvature vector of $L$, respectively.
Hence, $T(\boldsymbol{b}) \subset \mathbb C^n$ is Hamiltonian minimal.
Using his second variation formula \cite[Thorem 3.4]{Oh1993}, Oh proved that
the torus $T(\boldsymbol{b}) \subset \mathbb C^n$ is a Hamiltonian stable
Lagrangian submanifold (see \cite[Theorem 4.1]{Oh1993}).
Moreover, the isoperimetric inequality for closed curves in $\mathbb R^2$
states that $T(b_1) \subset {\mathbb C}$ is Hamiltonian volume minimizing.
Based on these results, Oh proposed the following

\begin{conj}[Oh \cite{Oh1993}, p.\,192] \label{conj:Oh} \rm
The Lagrangian torus $T(\boldsymbol{b})$ in $\mathbb C^n$ is Hamiltonian volume minimizing.
\end{conj}

In a sense, Conjecture \ref{conj:Oh} is regarded as a symplectic higher dimensional analogue
of the isoperimetric inequality in $\mathbb R^2$.
Though the statement is quite natural, it turned out to be false for $n \geq 3$.
Indeed, C.~Viterbo \cite[p.\,419]{Viterbo00} has already pointed out that
$T(1,2,2)$ and $T(1,2,3)$ are Hamiltonian isotopic based on a remarkable result by
Chekanov \cite[Theorem A]{Chekanov1996}, see Section 2.
Namely, the second one is not Hamiltonian volume minimizing.
Furthermore, in Section 2 we prove

\begin{cor} \label{cor:counterex}
Let $\boldsymbol{b} \in ({\mathbb R}_{>0})^n$.
If $N(\boldsymbol{b}) \geq 3$, then the Lagrangian torus
$T(\boldsymbol{b}) \subset \mathbb C^n$
is not Hamiltonian volume minimizing.
\end{cor}

If $n \geq 3$, then the set
$$
\{ \boldsymbol{b} \in ({\mathbb R}_{>0})^n \mid N(\boldsymbol{b}) \geq 3 \}
$$
is an open dense subset of $({\mathbb R}_{>0})^n$, and hence almost all product tori
in $\mathbb C^n\ (n \geq 3)$ are {\it not} Hamiltonian volume minimizing.
Notice that $T(\boldsymbol{b})$ is represented as $\mu_0^{-1}(b_1/2\pi,\ldots,b_n/2\pi)$, where
$$
\mu_0(x_1,\dots,x_n,y_1,\dots,y_n)=
\left(\frac12 (x_1^2+y_1^2),\dots, \frac12 (x_n^2+y_n^2)\right)
$$
is the moment map $\mu_0:{\mathbb C}^n \to (\mathbb{R}_{\ge 0})^n$
associated with the standard Hamiltonian action by the real torus
$T^n\subset ({\mathbb C}^\times)^n$ on $\mathbb{C}^n$.

\smallskip

Similarly, the complex projective space (${\mathbb C}P^n,J_{\rm std})$ equipped with
the standard Fubini-Study K\"ahler form $\omega_{\rm FS}$ admits an effective Hamiltonian
$T^n$-action.
Each principal orbit is a flat Lagrangian torus in ${\mathbb C}P^n$ like product one in
$\mathbb{C}^n$.
As for its Hamiltonian minimality and Hamiltonian stability,
the second author previously proved

\begin{pro}[\cite{Ono07}, Section 4] \label{pro:ono}
Any Lagrangian torus orbit $T^n$ in $({\mathbb C}P^n,\omega_{\rm FS},J_{\rm std})$ 
is Hamiltonian minimal and Hamiltonian stable.
\end{pro}

Hence, it is worthwhile to determine whether each Lagrangian torus orbit $T^n$
is Hamiltonian volume minimizing or not.
The following is the main result of the present article,
which provides a negative solution for the problem
(see Conjecture 1.4 in \cite{Ono07}).

\begin{thm} \label{thm:mainCPn}
If $ n \geq 3$, then
almost all Lagrangian torus orbits in ${\mathbb C}P^n$ are not Hamiltonian volume minimizing.
\end{thm}

The proof, which is given in Section 3,
is based on a recent result of Chekanov and Schlenk \cite{CS2014}
which gives a refinement of the Chekanov's one mentioned above.

\smallskip

In general, Darboux's theorem says that any point in a symplectic manifold $(M,\omega)$
possesses a neighbourhood which is isomorphic to a neighbourhood of the origin of
$(\mathbb C^n,\omega_0)$.
Then the Chekanov-Schlenk's theorem ensures any symplectic manifold the existence of
a pair of Lagrangian tori which are mutually Hamiltonian isotopic and not intersect.
Furthermore, in the class of compact toric symplectic manifolds,
we can regard the Chekanov-Schlenk's theorem as a local model of a $T^n$-fixed point
of such a manifold.
Although the result is weaker than the case of ${\mathbb C}P^n$,
this observation yields the following

\begin{thm} \label{thm:torusfibre}
Let $(M,\omega,J)$ be a complex $n$-dimensional compact toric K\"ahler manifold.
If $n \geq 3$, then there exists a toric fibre of $M$ (indeed, infinitely many)
which is not Hamiltonian volume minimizing.
\end{thm}

We prove it in Section 4.
Notice that toric fibres in Theorem \ref{thm:torusfibre} are all Hamiltonian minimal
(see Section 4).

\section{Product tori in $\mathbb C^n$ and Chekanov-Schlenk's Theorem}

In this section, we shall consider the case of $({\mathbb C}^n,\omega_0)$.
For $\boldsymbol{a}=(a_1,\dots,a_n)\in (\mathbb{R}_{>0})^n$, we use the following notations:
\begin{equation*}
\begin{split}
&\underline{\boldsymbol{a}}=\min\{a_i\,|\,1\le i\le n\},\ \ 
\overline{\boldsymbol{a}}=\max\{a_i\,|\,1\le i\le n\},\ \ 
\lvert \boldsymbol{a} \rvert =\sum_{i=1}^n a_i ,\\
&m(\boldsymbol{a})=\#\{i\,|\, a_i=\underline{\boldsymbol{a}}\},\ \ 
\| \boldsymbol{a} \| =\lvert \boldsymbol{a} \rvert +\underline{\boldsymbol{a}},\ \ 
\lvert \| \boldsymbol{a}\| \rvert =\lvert \boldsymbol{a}\rvert +\overline{\boldsymbol{a}},\\
&\Gamma(\boldsymbol{a})=\mathrm{span}_{\mathbb Z}
\langle a_1-\underline{\boldsymbol{a}},\dots,a_n-\underline{\boldsymbol{a}} \rangle
\subset {\mathbb R}.
\end{split}
\end{equation*}
For $\boldsymbol{a},\boldsymbol{a}'\in (\mathbb{R}_{>0})^n$, we denote
$\boldsymbol{a}\simeq \boldsymbol{a}'$ if
$$
(\underline{\boldsymbol{a}},m(\boldsymbol{a}),\Gamma(\boldsymbol{a}))
=(\underline{\boldsymbol{a}}',m(\boldsymbol{a}'),\Gamma(\boldsymbol{a}')),
$$
and consider the set
$$
\tilde {\Delta}_s:=\left\{(a_1,\dots,a_n)\in (\mathbb{R}_{\ge 0})^n\,
\Big|\,
\sum_{i=1}^na_i<s\right\}.
$$
Notice that $\mu_0^{-1}(\tilde {\Delta}_s)$ is the open ball in ${\mathbb C}^n$
with radius $\sqrt{2s}$ centred at the origin.
Let $L$ and $L'$ be Lagrangian submanifolds of $(M,\omega)$.
Then $L$ is said to be {\it Hamiltonian isotopic to} $L'$ if there exists 
$\phi \in \mathrm{Ham}_c(M,\omega)$ such that $\phi(L)=L'$.
The following result is fundamental for the arguments of this article.

\begin{thm}[Chekanov \cite{Chekanov1996}]\label{ThmCh}
Let $\boldsymbol{a},\boldsymbol{a}'\in (\mathbb{R}_{>0})^n$.
A product torus $T(\boldsymbol{a})$ of $(\mathbb C^n,\omega_0)$ is Hamiltonian isotopic to
$T(\boldsymbol{a}')$ if and only if $\boldsymbol{a} \simeq \boldsymbol{a}'$ holds.
\end{thm}

\begin{pro}[Corollary \ref{cor:counterex}] \label{pro:counterex}
If $N(\boldsymbol{a})\ge 3$, then the product torus
$\mu_0^{-1}(\boldsymbol{a})=T(2\pi\boldsymbol{a}) \subset {\mathbb C}^n$
is not Hamiltonian volume minimizing.
\end{pro}

\begin{proof}
For $\boldsymbol{a}=(a_1,\ldots,a_n) \in (\mathbb{R}_{>0})^n$, by assumption,
there exist numbers $i,j \in \{ 1,2,\ldots,n \}$ such that $\underline{\boldsymbol{a}}<a_i<a_j$.
We define a new $\boldsymbol{a}'$ as
$$
\boldsymbol{a}'=(a'_1,\ldots,a'_n)
:=(a_1,\ldots.a_{j-1},a_j-a_i+\underline{\boldsymbol{a}},a_{j+1},\ldots.a_n).
$$
Then we have
$\boldsymbol{a}\simeq \boldsymbol{a}'$ and $\|\boldsymbol{a}\|>\|\boldsymbol{a}'\|$.
Since $\Pi_i a_i > \Pi_i a'_i$,
Theorem \ref{ThmCh} implies that
$\mu_0^{-1}(\boldsymbol{a})$ is not Hamiltonian volume minimizing.
\end{proof}

Furthermore, the size of the support of a Hamiltonian isotopy connecting two product tori
in Theorem \ref{ThmCh} has precisely estimated as follows.
This estimation is essential to treat the case of ${\mathbb C}P^n$.

\begin{thm}[Chekanov-Schlenk \cite{CS2014}, Theorem 1.1] \label{thm:CS}
For $\boldsymbol{a},\boldsymbol{a}'\in (\mathbb{R}_{>0})^n$,
suppose that $\boldsymbol{a} \simeq \boldsymbol{a}'$.
Let $s$ be a positive number satisfying that $s>\max\{\|\boldsymbol{a}\|,\|\boldsymbol{a}'\|\}$.
Then there exists a smooth Hamiltonian function
$H:[0,1]\times \mathbb{C}^{n} \to \mathbb{R}$
satisfying the following:
\begin{enumerate}
\item $\mathrm{Supp }(H) \subset [0,1] \times \mu_0^{-1}(\Tilde{\Delta}_s)$.
\item $\phi_H^1(\mu_0^{-1}(\boldsymbol{a}))=\mu_0^{-1}(\boldsymbol{a}')$.
\end{enumerate}
\end{thm}

\section{Lagrangian torus orbits in ${\mathbb C}P^n$}

In this section, we shall treat the case of ${\mathbb C}P^n$ and prove the main theorem.

\subsection{${\bf e}_i$-action-angle coordinates}

Let us consider $\mathbb{R}^n$ and
take an orthonormal basis
\begin{equation*}
%{\bf e}_0:=
%\begin{pmatrix}
%0\\
%0\\
%\vdots\\
%0
%\end{pmatrix},\ \ 
{\bf e}_1:=
\begin{pmatrix}
1\\
0\\
\vdots\\
0
\end{pmatrix},\ \ 
{\bf e}_2:=
\begin{pmatrix}
0\\
1\\
\vdots\\
0
\end{pmatrix},\ \ 
\dots,\ \ 
{\bf e}_n:=
\begin{pmatrix}
0\\
0\\
\vdots\\
1
\end{pmatrix}
\end{equation*}
of  $\mathbb{R}^n$
and set $\Delta:=\{(a_1,\dots,a_n)\in (\mathbb{R}_{\ge 0})^n\,|\,\sum_{i=1}^na_i\le 1\}$.
For a notational reason, we put
${\bf e}_0:=\ ^{t}(0,0,\ldots,0) \in \mathbb{R}^n$.
The symplectic toric manifold corresponding to the polytope $\Delta$ is nothing but
the $n$-dimensional complex projective space $(\mathbb{C}P^n,\omega_{\rm FS},\mu)$.

We first examine the coordinate neighbourhood given by
\begin{equation*}
U_0=\{[z_0:z_1:\dots:z_n]\,|\, z_0\not=0\}\overset{\sim}{\to}
\mathbb{C}^n,\ \ 
[z_0:\dots:z_n]\mapsto (\frac{z_1}{z_0},\dots,\frac{z_n}{z_0}).
\end{equation*}
We put
\begin{equation*}
r_0^i:=\left\lvert \frac{z_i}{z_0}\right\rvert,\ \ 
\theta_0^i:=\mathrm{arg}\frac{z_i}{z_0}
\end{equation*}
for $i=1,\ldots,n$.
Then the moment map associated with the standard Hamiltonian $T^n$-action on
$(\mathbb{C}P^n,\omega_{\rm FS},\mu)$ is represented as
\begin{equation*}
\mu=\sum_{i=1}^nu_0^i{\bf e}_i,\ \ \ u_0^i:=\frac{(r_0^i)^2}{1+\sum_
{j=1}^n(r_0^j)^2}.
\end{equation*}
Here we introduce the coordinates defined by
\begin{equation*}
x_0^i:=\sqrt{2u_0^i}\cos \theta_0^i,\ \ 
y_0^i:=\sqrt{2u_0^i}\sin \theta_0^i.
\end{equation*}
Then, on $U_0$ the symplectic structure $\omega_{\rm FS}$
and the moment map $\mu$ are expressed as follows:
\begin{equation*}
\omega_{\rm FS}|_{U_0}=\sum_{i=1}^ndu_0^i\wedge d\theta_0^i=
\sum_{i=1}^ndx_0^i\wedge dy_0^i,\ \ 
\mu|_{U_0}=\frac12 \sum_{i=1}^n\{(x_0^i)^2+(y_0^i)^2\}{\bf e}_i.
\end{equation*}
Hence we have an isomorphism
\begin{equation*}
(U_0,\omega_{\rm FS}|_{U_0},\mu|_{U_0})
\cong (\mu_0^{-1}(\Tilde{\Delta}_1),\omega_0,\mu_0)
\end{equation*}
as Hamiltonian $T^n$-spaces.
We call the coordinates $(u_0^1,\ldots,u_0^n,\theta_0^1,\ldots,\theta_0^n)$
{\it ${\bf e}_0$-action-angle coordinates}.

Similarly, we examine the coordinate neighbourhood given by
\begin{equation*}
U_1=\{[z_0:z_1:\dots:z_n]\,|\, z_1\not=0\}\overset{\sim}{\to}
\mathbb{C}^n,\ \ 
[z_0:\dots:z_n]\mapsto (\frac{z_0}{z_1},\frac{z_2}{z_1},\dots,\frac{z_n}{z_1}).
\end{equation*}
(The case where $i\ge 2$ is similar.)
We put
\begin{equation*}
r_1^1:=\left\lvert \frac{z_0}{z_1}\right\rvert,\ \ 
r_1^i:=\left\lvert \frac{z_i}{z_1}\right\rvert\ (i\ge 2),\ \ 
\theta_1^1:=\mathrm{arg}\frac{z_0}{z_1},\ \ 
\theta_1^i:=\mathrm{arg}\frac{z_i}{z_1}\ (i\ge 2),
\end{equation*}
Then we have
\begin{equation*}
\mu=
u_1^1({\bf e}_0-{\bf e}_1)+\sum_{i=2}^nu_1^i({\bf e}_i-{\bf e}_1)+
{\bf e}_1
,\ \ \ u_1^i:=\frac{(r_1^i)^2}{1+\sum_
{j=1}^n(r_1^j)^2}\ (i\ge 1).
\end{equation*}
We also introduce the coordinates defined by
\begin{equation*}
x_1^i:=\sqrt{2u_1^i}\cos \theta_1^i,\ \ 
y_1^i:=\sqrt{2u_1^i}\sin \theta_1^i.
\end{equation*}
Then, on $U_1$ we have 
\begin{equation*}
\omega_{\rm FS}|_{U_1}=\sum_{i=1}^ndu_1^i\wedge d\theta_1^i=
\sum_{i=1}^ndx_1^i\wedge dy_1^i,
\end{equation*}
\begin{equation*}
\mu|_{U_1}=\frac12 \{(x_1^1)^2+(y_1^1)^2\}({\bf e}_0-{\bf e}_1)
+\frac12 \sum_{i=2}^n\{(x_1^i)^2+(y_1^i)^2\}({\bf e}_i-{\bf e}_1)+
{\bf e}_1.
\end{equation*}
Hence we obtain an isomorphism
\begin{equation*}
(U_1,\omega_{\rm FS}|_{U_1},\mu|_{U_1})
\cong (\mu_0^{-1}(\Tilde{\Delta}_1),\omega_0,\mu_0)
\end{equation*}
as Hamiltonian $T^n$-spaces.
Moreover, on $U_0\cap U_1$,
\begin{equation*}
r_0^1=\frac{1}{r_1^1},\ \ 
r_0^j=\frac{r_1^j}{r_1^1}\ (j\ge 2),\ \ 
\theta_0^1=-\theta_1^1,\ \ 
\theta_0^j=\theta_1^j-\theta_1^1\ (j\ge 2)
\end{equation*}
hold. Then we have
\begin{align*}
u_0^1=\frac{1}{1+\sum_{j=1}^n(r_1^j)^2}=1-\sum_{j=1}^nu_1^j,\ \ 
u_0^j=u_1^j\ (j\ge 2)
\end{align*}
and we can easily check that
\begin{align*}
\mu &=\sum_{j=1}^nu_0^j{\bf e}_j
=u_1^1({\bf e}_0-{\bf e}_1)+\sum_{i=2}^nu_1^i({\bf e}_i-{\bf e}_1)+{\bf e}_1.
\end{align*}
Similarly, the ${\bf e}_i$-action coordinates $(u_i^1,\ldots,u_i^n)$ satisfies that
$$
u_i^j=u_0^j\ (i\not=j),\ \ u_i^i=1-\sum_{j=1}^nu_0^j
$$
on $U_0\cap U_i$ and
\begin{equation*}
\mu|_{U_0 \cap U_i}=\sum_{j=1}^iu_i^j({\bf e}_{j-1}-{\bf e}_i)+
\sum_{j=i+1}^nu_i^{j}({\bf e}_{j}-{\bf e}_i)+{\bf e}_i.
\end{equation*}
Hence, on $U_0\cap U_i$ the symplectic structure $\omega_{\rm FS}$ described as
$$
\omega_{\rm FS}|_{U_0 \cap U_i}=\sum_{j=1}^ndu_0^j\wedge d\theta_0^j=
\sum_{j=1}^ndu_i^j\wedge d\theta_i^j.
$$

\subsection{Volume of a Lagrangian torus orbit in $\mathbb{C}P^n$}

Recall that the moment map $\mu:\mathbb{C}P^n \to \Delta$ is
associated with the standard Hamiltonian $T^n$-action on
$(\mathbb{C}P^n,\omega_{\rm FS})$.
The volume of a $T^n$-orbit $\mu^{-1}(p)$, $p\in \mathrm{Int}(\Delta)$,
can be calculated by using Abreu's symplectic potential (see Section 4).
Let $(u_0^1,\dots,u_0^n)$ be the ${\bf e}_0$-action coordinates of $p$.
Then we obtain
$$
(\vol(\mu^{-1}(p)))^2=C\left(1-\sum_{j=1}^nu_0^j\right)\prod_{k=1}^nu_0^k,
$$
where $C$ is a positive constant.
As for the ${\bf e}_i$-action coordinates $(u_i^1,\dots,u_i^n)$,
by the formula of the coordinate transformation examined in the previous subsection,
we have the same formula
\begin{equation} \label{eq:volume}
(\vol(\mu^{-1}(p)))^2=C\left(1-\sum_{j=1}^nu_i^j\right)\prod_{k=1}^nu_i^k.
\end{equation}

\subsection{Proof of the main theorem}

Here we give a property of a moment polytope which holds only for $\mathbb{C}P^n$
among compact toric K\"ahler manifolds.

\begin{lem} \label{lem:action}
Let ${\bf u}_i=(u_i^1,\dots,u_i^n)$ be the ${\bf e}_i$-action coordinates of
$p\in \mathrm{Int}(\Delta)$.
Then there exists a number $i$ such that $\lvert \| {\bf u}_i\| \rvert\le 1$.
\end{lem}
\begin{proof}
Suppose that $\lvert \| {\bf u}_0\| \rvert>1$.
By definition, there exists $i \in \{0,1,\dots,n\}$ such that $\overline{{\bf u}}_0=u_0^i$.
Then we have
\begin{eqnarray*}
{\bf u}_i
&=& (u_i^1,\dots,u_i^{i-1},u_i^i,u_i^{i+1},\dots,u_i^n) \\
&=& (u_0^1,\dots,u_0^{i-1},1-\lvert {\bf u}_0\rvert,u_0^{i+1},\dots,u_0^n),
\end{eqnarray*}
and hence $\lvert {\bf u}_i\rvert=1-u_0^i$.
Therefore,
\begin{equation*}
\lvert \| {\bf u}_i\| \rvert=
1-u_0^i+\overline{\bf u}_i=
\begin{cases}
1-u_0^i+u_0^j\le 1 & (\text{if} \ \ \overline{\bf u}_i=u_0^j\ (i\not=j))\\
2-\lvert \| {\bf u}_0\| \rvert<1 & (\text{if} \ \ \overline{\bf u}_i=1-\lvert {\bf u}_0\rvert),
\end{cases}
\end{equation*}
which implies $\lvert \| {\bf u}_i\| \rvert\le 1$.
\end{proof}

Theorem \ref{thm:mainCPn} is a direct consequence of the following

\begin{thm} \label{thm:main1}
Let $(\mathbb{C}P^n,\omega_{\rm FS})$ be the $n$-dimensional complex projective space.
Let $\mu:\mathbb{C}P^n \to \Delta$ be the moment map associated with the standard
Hamiltonian $T^n$-action on $\mathbb{C}P^n$.
Pick a point $p\in \mathrm{Int}(\Delta)$ and take ${\bf e}_i$-action coordinates
${\bf u}_i=(u_i^1,\dots,u_i^n)\in \mathrm{Int}(\Delta)$ of $p$
which satisfies that $\lvert \| {\bf u}_i\| \rvert\le 1$.
If $N({\bf u}_i)\ge 3$, then the Lagrangian torus orbit $\mu^{-1}(p) \subset \mathbb{C}P^n$
is not Hamiltonian volume minimizing.
\end{thm}

\begin{rem}
We denote by $D_n$ the set of all points in $\mathrm{Int}(\Delta)$
which satisfy the assumption of Theorem \ref{thm:main1}.
Of course, $D_n$ is open dense in $\mathrm{Int}(\Delta)$ if $n\geq3$.
\end{rem}

\begin{proof}
Firstly, since $\|{\bf u}_i\|<\lvert \| {\bf u}_i\| \rvert\le 1$,
by virtue of Theorem \ref{thm:CS},
we can take
$\boldsymbol{a}:={\bf u}_i$ in the proof of Proposition \ref{pro:counterex} and
there exist a positive number $\varepsilon>0$ and a smooth function
$H:[0,1]\times \mathbb{C}^n \to \mathbb{R}$
satisfying that
\begin{enumerate}
\item $\mathrm{Supp }(H)\subset [0,1] \times \mu_0^{-1}(\Tilde{\Delta}_{1-\varepsilon})$,
\item $\phi_H^1(\mu_0^{-1}({\bf u}_i))=\mu_0^{-1}({\bf u}_i')$.
\end{enumerate}
Let $p'$ be the element in $\mathrm{Int}(\Delta)$
whose ${\bf e}_i$-action coordinate is ${\bf u}_i'$.
Notice that we have the identification
$(U_i,\omega_{\rm FS}|_{U_i},\mu|_{U_i})
\cong (\mu_0^{-1}(\Tilde{\Delta}_1),\omega_0,\mu_0)$
as Hamiltonian $T^n$-spaces.
Denoting it by $\Phi: U_i \to \mu_0^{-1}(\Tilde{\Delta}_1) \subset \mathbb{C}^n$,
we can define the following Hamiltonian function on $\mathbb{C}P^n$:
$$
\hat{H}(t,x)
:=\left\{
  \begin{array}{ccl}
     H(t,\Phi(x)) &,& x \in U_i \\
     0 &,& x \in \mathbb{C}P^n \setminus U_i.
  \end{array}
 \right.
$$
Then $\hat{H} \in C^\infty_c([0,1] \times \mathbb{C}P^n)$ and we obtain
$\phi_{\hat{H}}^1(\mu^{-1}(p))=\mu^{-1}(p')$.

Secondly, let us compare their volume.
By assumption, there exist numbers $a,b \in \{1,\ldots,n\}$ such that
$\underline{{\bf u}}_i<u_i^a<u_i^b$.
Then by (\ref{eq:volume}) we obtain
\begin{align*}
&(\vol (\mu^{-1}(p)))^2-(\vol (\mu^{-1}(p')))^2\\
& =C\left(1-\sum_{j=1}^nu_i^j\right)\prod_{k=1}^nu_i^k \\
&\quad\quad -C\left(1-\sum_{j=1}^nu_i^j+u_i^b-(u_i^b-u_i^a+\underline{{\bf u}}_i)\right)
\frac{u_i^b-u_i^a+\underline{{\bf u}}_i}{u_i^b}
\prod_{k=1}^nu_i^k \\
&= C\frac{\prod_{k=1}^nu_i^k}{u_i^b}\left\{
u_i^b\left(1-\sum_{j=1}^nu_i^j\right)-\left(u_i^b-u_i^a+\underline{{\bf u}}_i
\right)\left(1-\sum_{j=1}^nu_i^j+u_i^a-\underline{{\bf u}}_i)\right)
\right\}\\
&=C\frac{\prod_{k=1}^nu_i^k}{u_i^b}(u_i^a-\underline{{\bf u}}_i)
\left(1-\sum_{j=1}^nu_i^j-u_i^b+u_i^a-\underline{{\bf u}}_i\right).
\end{align*}
Since
$\frac{\prod_{k=1}^nu_i^k}{u_i^b}(u_i^a-\underline{{\bf u}}_i)>0$
and
$$
1-\sum_{j=1}^nu_i^j-u_i^b+u_i^a-\underline{{\bf u}}_i
\ge 1-\lvert \| {\bf u}_i\| \rvert+u_i^a-\underline{{\bf u}}_i >0
$$
hold, we conclude that $\vol (\mu^{-1}(p))>\vol (\mu^{-1}(p'))$.
\end{proof}

\section{The case of toric K\"ahler manifolds}

In this section, we attempt to generalise the argument of Section 3 to toric K\"ahler manifolds.
From now on, let $(M,\omega,J)$ be a complex $n$-dimensional compact toric K\"ahler manifold,
i.e., $M$ admits an effective holomorphic action of
the complex torus $({\mathbb C}^\times)^n$ such that the restriction to
the real torus $T^n$ is Hamiltonian with respect to the K\"ahler form $\omega$.
Its moment map is denoted by $\mu:M\to \Delta=\mu(M) \subset {\mathbb R}^n$.
We may assume, without loss of generalities, that the moment polytope $\Delta$ satisfies
$$
\Delta=\{\boldsymbol{a}\in ({\mathbb R}_{\ge 0})^n \mid
l_r(\boldsymbol{a}):=\langle \boldsymbol{a},\mu_r\rangle-\lambda_r\ge 0,\ \lambda_r<0,\ 
r=n+1,\dots, d\},
$$
where each $\mu_r$ is a primitive element of the lattice ${\mathbb Z}^n \subset {\mathbb R}^n$
and inward-pointing normal to the $r$-th $(n-1)$-dimensional face of $\Delta$.
It is known that each fibre
$\mu^{-1}(\boldsymbol{a}),\ \boldsymbol{a} \in \mathrm{Int}(\Delta)$,
is a Lagrangian torus and Hamiltonian minimal (see, e.g., \cite[Proposition 3.1]{Ono07}).

The point $\mu^{-1}(0) \in M$ is a fixed point of the $(\mathbb C^\times)^n$-action.
By the construction, there exists a toric affine neighbourhood $U$ of $\mu^{-1}(0)$
such that $(U,\mu^{-1}(0) )$ is isomorphic to $(\mathbb C^n,0)$ as $(\mathbb C^\times)^n$-spaces.
Using this identification we can define the standard complex coordinates
$(w^1,\ldots,w^n)$ on $U$.
Their polar coordinates are given by $w^i=r^i e^{\sqrt{-1}\theta^i}$, $i=1,\ldots,n$.

As a set $U$ is described as
$$
U=M\setminus \mu^{-1}(\mathcal F),\ \ 
\mathcal F:=\bigcup_{F:\,\text{facet of}\,\Delta,\,0 \notin F}F.
$$
The restriction of the K\"ahler form $\omega$ on $U$ can be expressed as
$$
\omega_{|U}=2\sqrt{-1}\partial \overline{\partial}\varphi,
$$
where $\varphi$ is a real-valued function defined on $({\mathbb R}_{\ge 0})^n$
(see \cite{Abreu2003}, \cite{Guillemin1994}).
Then the moment map $\mu:M\to \Delta$ is represented as
$$
\mu(p)=\left(r^1\frac{\partial \varphi}{\partial r^1},\dots,
r^n\frac{\partial \varphi}{\partial r^n}\right)(p)=:(u^1,\dots,u^n)
$$
Putting $x^i:=\sqrt{2u^i}\cos\theta^i$ and $y^i:=\sqrt{2u^i}\sin\theta^i$,
a straightforward calculation yields
$$
\omega|_U=\sum_{i=1}^n dx^i\wedge dy^i=\sum_{i=1}^n du^i\wedge d\theta^i,\ 
\mu|_U=\frac12 \sum_{i=1}^n\{(x^i)^2+(y^i)^2\}{\bf e}_i
$$
on $U$.
Thus 
$(U,\omega_{|U},\mu_{|U})$ is isomorphic as Hamiltonian $T^n$-spaces to
$(V,\omega_0\,_{|V},\mu_0\,_{|V})$, where $V:=\mu_0^{-1}(\Delta\setminus \mathcal F)$
and $\mu_0$ is the moment map defined in Section 1.

Now we are in a position to prove our second result (Theorem \ref{thm:torusfibre}).

\begin{thm} \label{thm:torusfibre2}
Let $(M,\omega,J)$ be a complex $n$-dimensional compact toric K\"ahler manifold
equipped with the moment map $\mu:M\to \Delta\subset \mathbb{R}^n$
that is specified as above.
Assume that $n \geq 3$ and define a constant $s_0>0$ as
$$
s_0=\sup\{ s>0 \mid \tilde {\Delta}_s \subset \Delta \}.
$$
For $\boldsymbol{a} \in \mathrm{Int}(\Delta)$ with $N(\boldsymbol{a}) \geq 3$,
if $\| \boldsymbol{a} \| < s_0$, then
there exists $\boldsymbol{a}' \in \mathrm{Int}(\Delta)$ such that
$$\phi(\mu^{-1}(\boldsymbol{a}))=\mu^{-1}(\boldsymbol{a}')$$
for some $\phi \in \mathrm{Ham}(M,\omega)$.
Furthermore,
if $\| \boldsymbol{a} \|$ is sufficiently close to $0$, then
in addition these Lagrangian tori
$\mu^{-1}(\boldsymbol{a})$ and $\mu^{-1}(\boldsymbol{a}')$
satisfy
\begin{equation*}
\vol(\mu^{-1}(\boldsymbol{a}))>\vol(\mu^{-1}(\boldsymbol{a}')).
\end{equation*}
In particular, the above Lagrangian torus $\mu^{-1}(\boldsymbol{a})$
is not Hamiltonian volume minimizing in $M$.
\end{thm}

\begin{proof}
Given a vector $\boldsymbol{a} \in \mathrm{Int}(\Delta)$ satisfying that
$N(\boldsymbol{a}) \geq 3$ and $\| \boldsymbol{a} \| < s_0$,
according to Theorem \ref{thm:CS}, the proof of Proposition \ref{pro:counterex}
enables us to take $\boldsymbol{a}' \in \mathrm{Int}(\Delta)$
and $\{ \phi_H^t \}_{0 \leq t \leq 1} \subset \mathrm{Ham}_c({\mathbb C}^n,\omega_0)$
which satisfy
\begin{equation}\label{e6}
\phi_H^1(\mu_0^{-1}(\boldsymbol{a}))=\mu_0^{-1}(\boldsymbol{a}'), \quad
\mathrm{Supp}(H) \subset [0,1] \times \mu_0^{-1}(\tilde {\Delta}_{s_0})
\end{equation}
and
\begin{equation}\label{e7}
\prod_{i=1}^n a_i > \prod_{i=1}^n a'_i.
\end{equation}

Using the action angle coordinates $(u^1,\ldots,u^n,\theta^1,\ldots,\theta^n)$
on $U$ explained before,
we identify $(U,\omega_{|U},\mu_{|U})$ with $(V,\omega_0\,_{|V},\mu_0\,_{|V})$
and extend the Hamiltonian function $H$ on $U$ to $M$ as
$$
\hat{H}(t,x)
:=\left\{
  \begin{array}{ccl}
     H(t,x) &,& x \in U \\
     0 &,& x \in M \setminus U.
  \end{array}
 \right.
$$
Then $\hat{H} \in C^\infty([0,1] \times M)$ and hence we obtain
$\phi_{\hat{H}}^1(\mu^{-1}(\boldsymbol{a}))=\mu^{-1}(\boldsymbol{a}')$.

In order to complete the proof of Theorem \ref{thm:torusfibre},
we have to compare the volume of two flat tori
$\mu^{-1}(\boldsymbol{a})$ and $\mu^{-1}(\boldsymbol{a}')$
with respect to the induced metric from the toric K\"ahler manifold $(M,\omega,J)$.

In general, all $\omega$-compatible toric complex structures on $(M,\omega)$
can be parametrized by smooth functions on $\mathrm{Int}(\Delta)$,
which is shown by Abreu in \cite[Section 2]{Abreu2003}.
More precisely, we can choose a strictly convex function
$g \in C^\infty(\mathrm{Int}(\Delta))$ whose Hessian $\mathrm{Hess}_x(g)$
describes the complex structure $J$ on $M$,
and the determinant of $\mathrm{Hess}_x(g)$ is given by
$$
\left\{
\delta(x) \prod_{r=1}^d l_r(x)
\right\}^{-1},
$$
where $\delta \in C^\infty(\Delta)$ is a strictly positive function
(see \cite[Theorem 2.8]{Abreu2003}).
Then the Riemannian metric of the fibre $\mu^{-1}(p) \subset M$ of $p \in \mathrm{int}(\Delta)$
is given by the $(n \times n)$-matrix $(\mathrm{Hess}_x(g))^{-1}$,
and hence
$$
\vol(\mu^{-1}(\boldsymbol{a}))^2
=(2\pi)^{2n}\delta(\boldsymbol{a})\prod_{i=1}^n a_i
\prod_{r=n+1}^dl_r(\boldsymbol{a})
$$
holds.
However, in general it is difficult to compare 
$\vol(\mu^{-1}(\boldsymbol{a}))$ with $\vol(\mu^{-1}(\boldsymbol{a}'))$
from this expression.
So we introduce a parameter $c \in (0,1]$ and consider the volume of a Lagrangian torus
$\mu^{-1}(c \boldsymbol{a})$.
Then we obtain
\begin{multline*}
\vol(\mu^{-1}(c \boldsymbol{a}))^2-\vol(\mu^{-1}(c \boldsymbol{a}'))^2 \\
=
(2 \pi \sqrt{c})^{2n}\left\{\delta(c \boldsymbol{a})\prod_{i=1}^n a_i
\prod_{r=n+1}^dl_r(c \boldsymbol{a})-
\delta(c \boldsymbol{a}')\prod_{i=1}^n a'_i
\prod_{r=n+1}^dl_r(c \boldsymbol{a}')\right\}.
\end{multline*}
The value at $c=0$ of the quantity of the inside of the brackets is
\begin{equation*}
\delta(\boldsymbol{0})\prod_{r=n+1}^d(-\lambda_r)\left(
\prod_{i=1}^na_i-\prod_{i=1}^na'_i
\right),
\end{equation*}
which is positive due to (\ref{e7}).
Therefore, there exists a constant $c_{\boldsymbol{a}}>0$ such that
$$
\vol(\mu^{-1}(c \boldsymbol{a}))-
\vol(\mu^{-1}(c \boldsymbol{a}'))>0.
$$
holds for any $c\in (0,c_{\boldsymbol{a}})$.
Thus we complete the proof.
\end{proof}

\section{Remained open problems}

Finally, let us discuss the remained part of Oh's conjecture and add some remarks.
According to Corollary \ref{cor:counterex}, the unsolved part of Conjecture \ref{conj:Oh}
is as follows.

\begin{prob} \rm
Let $0<a \leq b$ and $k=1,2,\ldots n$.
Is a product torus
$T(\underbrace{a,\ldots,a}_{k},\underbrace{b,\ldots,b}_{n-k})$ in ${\mathbb C}^n$
Hamiltonian volume minimizing?
\end{prob}

This problem had already been considered by Anciaux in the case where $n=2$,
and he gave a partial answer to it.
He showed in \cite[Main Theorem]{Anciaux2002} that
$T(a,a) \subset {\mathbb C}^2$ has the least volume among all {\it Hamiltonian minimal}
Lagrangian tori of its Hamiltonian isotopy class.
However, this result does not imply that $T(a,a)$ is Hamiltonian volume minimizing in
${\mathbb C}^2$.

\smallskip

Next we turn to the case of ${\mathbb C}P^n$.
We proved in Theorem \ref{thm:mainCPn} that every Lagrangian torus
that is the preimage of a point in $D_n \subset \mathrm{Int}(\Delta)$ is not
Hamiltonian volume minimizing.
However, the barycentre $p_0$ of $\Delta$ is {\it not} in $D_n$.
The corresponding fibre $\mu^{-1}(p_0) \subset {\mathbb C}P^n$ is a minimal Lagrangian
torus and called the {\it Clifford torus}.
Thus the following question raised by Oh is still open.

\begin{prob}[\cite{Oh1990}, p.~516] \rm
Is the Clifford torus in ${\mathbb C}P^n$ Hamiltonian volume minimizing?
\end{prob}

We point out that Urbano proved that the only Hamiltonian stable minimal Lagrangian torus in
${\mathbb C}P^2$ is the Clifford one (see \cite[Corollary 2]{Urbano1993}).

\smallskip
\smallskip
\smallskip
\smallskip

\begin{flushleft}
Hiroshi Iriyeh

{\sc Mathematics and Informatics, College of Science, Ibaraki University\\
Mito, Ibaraki 310-8512, Japan}

{\it e-mail} : {\tt hiroshi.irie.math@vc.ibaraki.ac.jp}
\end{flushleft}

\begin{flushleft}
Hajime Ono

{\sc Department of Mathematics, Saitama University, 255 Shimo-Okubo, Sakura-Ku,\\
Saitama 380-8570, Japan}

{\it e-mail} : {\tt hono@rimath.saitama-u.ac.jp}

\end{flushleft}
\end{document}